\documentclass[leqno]{aadmbook}
\usepackage{amsthm}
\usepackage{amsfonts}
\usepackage{amsmath}
\usepackage{amsmath, amssymb, amsthm}
\usepackage{esint}
\usepackage{setspace}
\usepackage{multicol}
\usepackage{graphicx}
\usepackage{blindtext}
\usepackage{dirtytalk}
\usepackage{mathtools}

\graphicspath{ {Pictures/} }

\usepackage{hyperref}

\usepackage[font={small,it}]{caption}

\makeatletter
\newcommand*\bigcdot{\mathpalette\bigcdot@{.5}}
\newcommand*\bigcdot@[2]{\mathbin{\vcenter{\hbox{\scalebox{#2}{$\m@th#1\bullet$}}}}}
\makeatother

\newcommand{\vc}[1]{{\textbf{#1}}}
\newcommand{\fr}[1]{{\mathfrak{#1}}}
\newcommand{\B}{\mathcal{B}}

\theoremstyle{definition}

\newtheorem{theorem}{Theorem}[section]
\newtheorem{proposition}[theorem]{Proposition}

\newtheorem{lemma}[theorem]{Lemma}
\newtheorem{definition}{Definition}[section]

\makeatletter
\@addtoreset{proofpart}{theorem}
\makeatother

\begin{document}

%\maketitle

\oddsidemargin 16.5mm
\evensidemargin 16.5mm

\thispagestyle{plain}

\begin{center}

{\large\bf  MULTIVARIATE BELL POLYNOMIALS AND DERIVATIVES OF COMPOSED FUNCTIONS
\rule{0mm}{6mm}\renewcommand{\thefootnote}{}%Enter at least one, but not more than 3 MSCs.
% First entered MSC will be a primary one, others (at most 2) will be secondary.
\footnotetext{\scriptsize 2010 Mathematics Subject Classification. 26B05, 26B12.

\rule{2.4mm}{0mm}Keywords and Phrases. 
%Fill Keywords and Phrases Here - at most 5.
Bell polynomials, Fa\'a di Bruno formula, multivariate chain rule, repeated derivatives, composed multivariate functions
}}

\vspace{1cc}
{\large\it Aidan Schumann}

\vspace{1cc}
\parbox{24cc}{{\small

%Short abstract here (not more than 7 lines of text)
How do we take repeated derivatives of composed multivariate functions? for one-dimensional functions, the common tools consist of the Fa\'a di Bruno formula with Bell polynomials; while there are extensions of the Fa\'a di Bruno formula, there are no corresponding Bell polynomials. In this paper, we generalize the single-variable Bell polynomials to take vector-valued arguments indexed by multi-indices which we use to rewrite the Fa\'a di Bruno formula to find derivatives of \(\vc{f}(\vc{g}(\vc{x}))\).  
%Avoid citations in the  abstract, but if they are necessary, do not use labels from the
%main body of the article. For example, instead of  writing "formula (2)" , you
%should reproduce formula (2), or describe it in words. Instead of writing
%"the result  obtained in [1]" , you should write "the result obtained in
%{\sc J. A. Baker:}  {\it Isometries in normed spaces,}
%Amer. Math. Monthly {\bf 78} (1971), 655--658".

}}
\end{center}

\vspace{1cc}

% The paper should have at least two sections

%\vspace{1.5cc}
%\begin{center}
%{\bf 1. FILL SECTION TITLE HERE}
%\end{center}

\section{Introduction}
\label{sect:intro}

It is useful to have an expression for taking repeated derivatives of two composed multivariate functions such as \begin{equation} \label{eqn:subject}
\partial_{x_1}^{n_1}\partial_{x_2}^{n_2}\cdots\partial_{x_d}^{n_d} \vc{f}(\vc{g}(\vc{x})),
\end{equation}
but the tools for describing these derivatives are limited. In this paper, we fill this niche. 
To build to the more general multivariate case, we first review the method for computing repeated derivatives of composed single-variable functions such as
\begin{equation*}
   \frac{d^n}{dx^n} f(g(x)).
\end{equation*} 

While the first few derivatives are easy to compute by hand, the later terms need to be calculated with repeated use of the chain and product rules: \begin{equation} \label{eqn:FaaExample}
\begin{split}
\frac{d}{dx} f(g(x))     &= f'(g(x))\cdot g'(x); \\	
\frac{d^2}{dx^2} f(g(x)) &= f'(g(x))\cdot g''(x) + f''(g(x))\cdot g'(x)^2;\\
\frac{d^3}{dx^3} f(g(x)) &= f'(g(x))\cdot g'''(x) + 3f''(g(x))\cdot g'(x)g''(x) + f'''(g(x))\cdot g'(x)^3.
\end{split}
\end{equation} 
This only gets more complicated for multivariate functions. Clearly we need a better tool.

The prototypical treatment for finding the \(n\)-th derivative of composed single-variable functions is the combinatorial approach as found in Fa\'a di Bruno's paper and its translation \cite{arbel_fa`di_2016,bruno_sullo_1855}, which give the equation \begin{multline} \label{eqn:1DCombFaa}
\frac{d^n}{dx^n} f(g(x)) = \sum 
	\frac{n!}{k_1! k_2! \ldots k_n!}
	f^{(k)}(g(x)) \cdot
	\\
	\Big(\frac{g^{(1)}(x)}{1!}\Big)^{k_1}
	\Big(\frac{g^{(2)}(x)}{2!}\Big)^{k_2}
	\ldots
	\Big(\frac{g^{(n)}(x)}{n!}\Big)^{k_n}
\end{multline}
such that we sum over all terms of the ordered \(n\)-tuple \((k_1, k_2, \ldots, k_n)\) of non-negative integers such that 
\begin{equation*}
	n = 1k_1 + 2k_2 + 3k_3 + \ldots + nk_n
\end{equation*}
and \(k\) is defined to be 
\begin{equation*}
	k = k_1 + k_2 + \ldots + k_n.
\end{equation*}
Unfortunately, this method is cumbersome and not ideal for algebraic manipulations; we need a compact form. The most convenient simplification of the single-variable Fa\'a di Bruno formula uses partial exponential Bell polynomials, but there is no such equation for composed multivariate functions. In this paper, we generalize Bell polynomials and the Fa\'a di Bruno formula to apply these results to find a compact equation for taking repeated derivatives of multivariate functions such as (\ref{eqn:subject}).

\vspace{0.5 em}

\noindent\textbf{Roadmap:} We begin by summarizing past results about Bell polynomials (\S 2) and the single-variable Fa\'a di Bruno formula (\S 3) before defining multi-indices (\S 4). We then move into original work where we first define the multivariable Bell polynomial (\S 5) and generalize standard single-variable theorems to the multivariate case before using these multivariate Bell polynomials to simplify and prove the multivariate Fa\'a di Bruno formula (\S 6).

%\newpage
\section{Single-Variable Bell Polynomials}\label{sect:singleBell}

%%%%%%%%%%%%%%%%%%%%%%%%%%%%%%%%%%%%%%%%%%%%%%%
%\subsection{Single-Variable Bell Polynomials}%
%%%%%%%%%%%%%%%%%%%%%%%%%%%%%%%%%%%%%%%%%%%%%%%

Exponential Bell polynomials are so-called partition polynomials that characterize the ways of partitioning a set of size \(n\) into unordered parts \cite{charalambides_enumerative_2002,frabetti_five_2011}. The polynomials take indexed variables \(x_1,x_2, \ldots, x_n\) and give a multivariate polynomial which uses its terms and coefficients to describe the ways of partitioning the set. The index of the variable \(x_i\) is interpreted as the size of a partition set and the power to which each variable is taken is the number of parts of that size. Finally, the coefficient for each term gives the number of ways a partition of that size can be formed \cite{charalambides_enumerative_2002,frabetti_five_2011}. For instance, the term \(6x_1^2x_2\) means there are \(6\) ways of partitioning a set of size \(4\) into two sets of size \(1\) and one set of size \(2\). 
%We can see from Table \ref{tab:SetPart} that for a set of size four, there is 
%\begin{itemize}
%\item[--] one partition of size \(4\), % (\(x_4\)), 
%\item[--] four partitions of size \(1+3\), % (\(4x_1x_3\)), 
%\item[--] three partitions of size \(2+2\), % (\(3x_2^2\)), 
%\item[--] six partitions of size \(1+1+2\), % (\(6x_1^2x_2\)), 
%\item[--] one partition of size \(1+1+1+1\). % (\(x_1^4\)). 
%\end{itemize}
%Comparing this to the complete exponential Bell polynomial \(B_4(x_1, x_2, x_3,x_4)\) in Table \ref{tab:BellPoly}, we see the same information. 
Comparing this as well as the whole of Table \ref{tab:SetPart} to  Table \ref{tab:BellPoly}, we see the same information.

In this paper, we exclusively deal with exponential Bell polynomials, so for brevity, we drop the ``exponential'' and refer to them simply as ``Bell polynomials.'' There is also a finer distinction: complete Bell polynomials and partial Bell polynomials \cite{charalambides_enumerative_2002}. For both of these varieties, the above interpretation of the terms holds, however, there is a difference in which terms are included. Complete Bell polynomials are characterized solely by the size \(n\) of the set to be partitioned and describe every possible partition. Partial Bell polynomials are characterized by a second number as well: the number of parts \(k\) into which the original set is being partitioned \cite{charalambides_enumerative_2002}.

\begin{table}[]
\begin{center}
\scriptsize
%\footnotesize
\begin{tabular}{rllll}
\multicolumn{1}{l}{}                 & \(k=1\)                              & \(k=2\)                                 & \(k=3\)                                    & \(k=4\)                                       \\ \cline{2-5} 
\multicolumn{1}{r|}{\(\{a\}\)}       & \multicolumn{1}{l|}{\(\{a\}\)}       & \multicolumn{1}{l|}{}                   & \multicolumn{1}{l|}{}                      & \multicolumn{1}{l|}{}                         \\ \cline{2-5} 
\multicolumn{1}{r|}{\(\{a,b\}\)}     & \multicolumn{1}{l|}{\(\{a,b\}\)}     & \multicolumn{1}{l|}{\(\{a\}\{b\}\)}     & \multicolumn{1}{l|}{}                      & \multicolumn{1}{l|}{}                         \\ \cline{2-5} 
\multicolumn{1}{r|}{\(\{a,b,c\}\)}   & \multicolumn{1}{l|}{\(\{a,b,c\}\)}   & \multicolumn{1}{l|}{\(\{a\}\{b,c\}\)}   & \multicolumn{1}{l|}{\(\{a\}\{b\}\{c\}\)}   & \multicolumn{1}{l|}{}                         \\
\multicolumn{1}{r|}{}                & \multicolumn{1}{l|}{}                & \multicolumn{1}{l|}{\(\{b\}\{a,c\}\)}   & \multicolumn{1}{l|}{}                      & \multicolumn{1}{l|}{}                         \\
\multicolumn{1}{r|}{}                & \multicolumn{1}{l|}{}                & \multicolumn{1}{l|}{\(\{c\}\{a,b\}\)}   & \multicolumn{1}{l|}{}                      & \multicolumn{1}{l|}{}                         \\ \cline{2-5} 
\multicolumn{1}{r|}{\(\{a,b,c,d\}\)} & \multicolumn{1}{l|}{\(\{a,b,c,d\}\)} & \multicolumn{1}{l|}{\(\{a\}\{b,c,d\}\)} & \multicolumn{1}{l|}{\(\{a,b\}\{c\}\{d\}\)} & \multicolumn{1}{l|}{\(\{a\}\{b\}\{c\}\{d\}\)} \\
\multicolumn{1}{r|}{}                & \multicolumn{1}{l|}{}                & \multicolumn{1}{l|}{\(\{b\}\{a,c,d\}\)} & \multicolumn{1}{l|}{\(\{a,c\}\{b\}\{d\}\)} & \multicolumn{1}{l|}{}                         \\
\multicolumn{1}{r|}{}                & \multicolumn{1}{l|}{}                & \multicolumn{1}{l|}{\(\{c\}\{a,b,d\}\)} & \multicolumn{1}{l|}{\(\{a,d\}\{b\}\{c\}\)} & \multicolumn{1}{l|}{}                         \\
\multicolumn{1}{r|}{}                & \multicolumn{1}{l|}{}                & \multicolumn{1}{l|}{\(\{d\}\{a,b,c\}\)} & \multicolumn{1}{l|}{\(\{b,c\}\{a\}\{d\}\)} & \multicolumn{1}{l|}{}                         \\
\multicolumn{1}{r|}{}                & \multicolumn{1}{l|}{}                & \multicolumn{1}{l|}{\(\{a,b\}\{c,d\}\)} & \multicolumn{1}{l|}{\(\{b,d\}\{a\}\{c\}\)} & \multicolumn{1}{l|}{}                         \\
\multicolumn{1}{r|}{}                & \multicolumn{1}{l|}{}                & \multicolumn{1}{l|}{\(\{a,c\}\{b,d\}\)} & \multicolumn{1}{l|}{\(\{c,d\}\{a\}\{b\}\)} & \multicolumn{1}{l|}{}                         \\
\multicolumn{1}{r|}{}                & \multicolumn{1}{l|}{}                & \multicolumn{1}{l|}{\(\{a,d\}\{b,c\}\)} & \multicolumn{1}{l|}{}                      & \multicolumn{1}{l|}{}                         \\ \cline{2-5} 
\end{tabular}
\end{center}
\caption{The partitions of sets of size 1, 2, 3, and 4. The order of the parts does not matter. The number \(k\) indicates the number of nonempty partitions.}\label{tab:SetPart}
\end{table}

\begin{table}[]
\begin{center}
\begin{tabular}{llllllll}
                                     & \(k=1\) &   & \(k=2\)              &   & \(k=3\)       &   & \(k=4\)   \\ \cline{2-8} 
\multicolumn{1}{r|}{\(B_1(x_1)=\)} & \(x_1\) &   &                      &   &               &   &           \\
\multicolumn{1}{r|}{\(B_2(x_1,x_2)=\)} & \(x_2\) & + & \(x_1^2\)            &   &               &   &           \\
\multicolumn{1}{r|}{\(B_3(x_1,x_2,x_3)=\)} & \(x_3\) & + & \(3 x_1 x_2\)          & + & \(x_1^3\)     &   &           \\
\multicolumn{1}{r|}{\(B_4(x_1,x_2,x_3,x_4)=\)} & \(x_4\) & + & \(4x_1x_3 + 3x_2^2\) & + & \(6x_1^2x_2\) & + & \(x_1^4\)
\end{tabular}
\end{center}

\caption{The first four complete exponential Bell polynomials with their corresponding partial polynomials. Reading a row in its entirety gives a complete exponential Bell polynomial while reading a single entry gives the \(k\)-th partial exponential Bell polynomial corresponding to the complete exponential Bell polynomial of that row.}\label{tab:BellPoly}
\end{table}

The complete Bell polynomials are defined using a sum indexed by the solution set of \(n\)-tuples \((k_1, \ldots, k_n)\) defined to be  \begin{equation} \label{eqn:1DCompleteSolSet}
K_n = \Big\{ (k_1, \ldots, k_n) \Big| \sum_{j=1}^n j k_j = n \Big\}
\end{equation}
which is used to define the \(n\)-th complete Bell polynomial
\begin{equation*}
   B_n(x_1, \ldots, x_n) 
   = 
   n! 
   \sum_{k_j \in K_n} \prod_{j=1}^n 
      \frac{1}{k_j!} 
      \left(
         \frac{x_j}{j!}
      \right)^{k_j},
\end{equation*}
see \cite{bell_exponential_1934,charalambides_enumerative_2002}. 
While for the single-variable Bell polynomials the above notation is sufficient, in this paper we need the more compact notation 
\begin{equation*}
B_n(x_j ; j) = B_n(x_1, \ldots, x_n)
\end{equation*}
where \(x_j\) denotes the variables \(x_1, x_2, \ldots, x_n\) and the \(j\) after the semicolon makes explicit the index over which we take the product to form the polynomial.

Similarly for partial Bell polynomials, the solution sets are defined to be the set of \(n-k+1\)-tuples \begin{equation} \label{eqn:1DPartialSolSet}
K_{n,k} = \Big\{
(k_1, \ldots, k_{n-k+1}) \Big| \sum_{j=1}^{n-k+1} j k_j = n \textrm{ and } \sum_{j=1}^{n-k+1} k_j = k \Big\},
\end{equation}
which is used to define the partial exponential Bell polynomial
\begin{equation*}
B_{n,k}(x_1,\ldots,x_{n-k+1}) = 
n! 
\sum_{k_j \in K_{n,k}} 
	\prod_{j=1}^{n-k+1}
		\frac{1}{k_j!} \left(\frac{x_j}{j!}\right)^{k_j},
\end{equation*}
see \cite{bell_exponential_1934,charalambides_enumerative_2002}. As before, we write a more compact form 
\begin{equation*}
B_{n,k}(x_j ; j) = B_{n,k}(x_1,\ldots,x_{n-k+1})
\end{equation*}

While in these definitions it is convenient to treat \(k_j\) as indices in a \(n\)- or \(n-k+1\)-tuple, this interpretation becomes a hinderance when we generalize the Bell polynomials to higher dimensions. An alternate interpretation of \(k_j\) is as a function \(k_j:\mathbb{N}\setminus\{0\} \rightarrow \mathbb{N}\), where \(\mathbb{N}\) is the set of nonnegative integers, which maps \(j \mapsto k(j)=k_j\). The defining conditions of the solution sets (\ref{eqn:1DCompleteSolSet}) and (\ref{eqn:1DPartialSolSet}) guarantee that \(k_j = 0\) for \(j>n\) in complete Bell polynomials and for \(j> n-k+1\) in partial Bell polynomials \cite{charalambides_enumerative_2002}. Using this fact we define the partial and complete bell polynomials in a compact way.

\begin{definition}[Partial and complete Bell polynomials] \label{def:1DBell}
Let \(K_n\) and \(K_{n,k}\) be the solution sets defined as 
\begin{align*}
K_n &= \Big\{ k_j \Big| \sum_{j=1}^\infty j k_j = n \Big\}
\\
K_{n,k} &= \Big\{
k_j \Big| \sum_{j=1}^{\infty} j k_j = n \textrm{ and } \sum_{j=1}^{\infty} k_j = k \Big\}.
\end{align*}
Then the complete and partial Bell polynomials are, respectively, defined to be \begin{align*}
B_{n}(x_j;j) &= 
n! 
\sum_{k_j \in K_{n}} 
	\prod_{j=1}^{\infty}
		\frac{1}{k_j!} \left(\frac{x_j}{j!}\right)^{k_j},
\\
B_{n,k}(x_j;j) &= 
n! 
\sum_{k_j \in K_{n,k}} 
	\prod_{j=1}^{\infty}
		\frac{1}{k_j!} \left(\frac{x_j}{j!}\right)^{k_j}.
\end{align*}
\end{definition}

Recall that, while the complete Bell polynomials describe all the partitions of a set, the partial Bell polynomial \(B_{n,k}\) give the number and type of partitions of a set of size \(n\) into \(k\) partitions. Thus by summing over all partial Bell polynomials with constant \(n\), we get the corresponding complete Bell polynomial \begin{equation} \label{eqn:1Dpartcomb}
B_n(x_j ; j) = \sum_{k=1}^n B_{n,k}(x_j; j),
\end{equation}
see \cite{charalambides_enumerative_2002}.

%\newpage
\section{The Single-Variable Fa\'a di Bruno Formula}
\label{sect:singleFaa}
%%%%%%%%%%%%%%%%%%%%%%%%%%%%%%%%%%%%%%%%%%%%%
%\subsection{Single-Variable Fa\'a di Bruno}%
%%%%%%%%%%%%%%%%%%%%%%%%%%%%%%%%%%%%%%%%%%%%%

The partial Bell polynomials sequester the combinatorial components of the Fa\'a di Bruno formula and allow for easier manipulation \cite{riordan_derivatives_1946}. Applying the definition of partial Bell polynomials (Definition \ref{def:1DBell}) to the Fa\'a di Bruno formula (\ref{eqn:1DCombFaa}) gives 
\begin{equation} 
\label{eqn:1DFaaBell}
\frac{d^n}{dx^n}f(g(x)) = \sum_{k=1}^n f^{(k)}(g(x)) B_{n,k}(g^{(j)}(x);j),
\end{equation}
see \cite{charalambides_enumerative_2002,frabetti_five_2011}
Comparing Table \ref{tab:BellPoly} to (\ref{eqn:FaaExample}) we confirm (\ref{eqn:1DFaaBell}) for the first few derivatives. Using this formulation, one can derive an elegant form for the Lagrange inversion theorem, as well as find elements of diffeomorphism groups \cite{charalambides_enumerative_2002,frabetti_five_2011}. 

There have been further generalizations to the Fa\'a di Bruno formula. Significant work has been done to find an elegant form for the repeated derivative of the composition of an arbitrary number of single-variable functions. Much of this work was done by Natalini and Ricci, who derived a formula for the higher-order Bell polynomials and applied it to the Fa\'a di Bruno formula for the repeated derivative of the composition of an arbitrary number of functions \cite{natalini_extension_2004,natalini_remarks_2016, natalini_higher_2017}. 

Furthermore, there has been significant work done to find the repeated derivative of \(f(\vc{x}(t))\) where \(\vc{x}(t)\) is a path through some \(d\)-dimensional space. This treatment too has received an elegant Bell polynomial representation using a generalization of Bell polynomials \cite{bernardini_multidimensional_2005, mishkov_generalization_2000, noschese_differentiation_2003}.

More relevantly to this paper, there has been work to derive a combinatorial treatment of the full multivariate Fa\'a di Bruno formula. The papers which provide this are works by Constantine and by Ma \cite{constantine_multivariate_1996, ma_higher_2009}. These previous works focus on finding a combinatorial form of the multivariate Fa\'a di Bruno formula analogous to (\ref{eqn:1DCombFaa}) for composed multivariate functions such as (\ref{eqn:subject}) and have not found a form analogous to (\ref{eqn:1DFaaBell}).

While there has been significant work to derive a combinatorial form for these derivatives, there has yet to be a corresponding generalization to  the Bell polynomial found to simplify the full multivariate Fa\'a di Bruno formula. In this paper, we present a multivariate Bell polynomial and apply it to rewrite and find a novel proof for the multivariate Fa\'a di Bruno formula.

%\subsection{Roadmap}

%We begin by introducing notation which simplifies our expressions in Section \ref{sect:conventions} before defining the multivariate Bell polynomial as well as providing useful identities in Section \ref{sect:bell} and finish by proving a compact form of the multivariate Fa\'a di Bruno formula in Section \ref{sect:bruno}.

\newpage
\section{Conventions and Multi-Indices}
\label{sect:conventions}

In this paper, we make extensive use of vectors and multi-indices to shorten and simplify equations. We use a subscript to indicate to which component we refer; for example, \(\vc{n} = (n_1, n_2, n_3, \ldots, n_d) \in \mathbb{N}^d\) such that \(\mathbb{N}\) is the set of natural numbers with \(0\). Such vectors over the natural numbers are so-called multi-indices \cite{Folland_Multi}. We use \(\vc{e}\) to denote a unit vector in \(\mathbb{N}^d\) with a subscript to refer to the position of the \(1\) and we use \(\vc{1}\), to denote the multi-index with a \(1\) as every component.

We use the following standard operations to manipulate the multi-indices \cite{Folland_Multi}.
\begin{definition}[Multi-index operations]
\label{def:veclen} If \(\vc{n} \in \mathbb{N}^d\), then 
\begin{equation*}
   |\vc{n}| = \sum_{i = 1}^d |n_i| = \sum_{i = 1}^d n_i 
   \hspace{3 em}
   \textrm{and}
   \hspace{3 em}
   \vc{n}! = \prod_{i=1}^d n_i !.
\end{equation*}
\end{definition}

In order for our definition of the multivariate Bell polynomial to be as general as possible, we define it on a vector space over an abstract field \(\mathbb{F}\). We need a way to compress exponents of components of vectors over \(\mathbb{F}\). \begin{definition} [Multi-index exponent]
\label{def:vecexp}Let \(\mathbb{F}\) be a field, \(\mathbb{F}^d\) a \(d\)-dimensional vector field over \(\mathbb{F}\), and \(\vc{x} = (x_1, x_2, x_3, \ldots, x_d) \in \mathbb{F}^d\). Let \(\vc{n} \in \mathbb{N}^d\). Then vector exponentiation is defined as 
\begin{equation*}
\vc{x}^\vc{n} = \prod_{i=1}^d (x_i)^{n_i}.
\end{equation*}
\end{definition}

For the derivative, it does not matter if we use the real or complex numbers in order to define our derivative so long as the derivative obeys the multivariate chain rule and the product rule. Thus, instead of writing \(\mathbb{R}\) or \(\mathbb{C}\), we again write \(\mathbb{F}\). Finally, in order to compress the derivative, we employ a multi-index form of Newton and Leibniz notation.

\begin{definition} [Multi-index derivative]
\label{def:vecder}
To simplify expressions with multivariate derivatives, we define the multi-index derivative to be: 
\begin{equation*}
   \partial^{n_1}_{x_1} \partial^{n_2}_{x_2} \ldots \partial^{n_d}_{x_d} f(\vc{x})
   =
   \partial_\vc{x}^\vc{n}f(\vc{x})
   =
   f^{(\vc{n})}(\vc{x}).
\end{equation*}
We use \(\partial_\vc{x}^\vc{n}f(\vc{x})\) and \(f^{(\vc{n})}(\vc{x})\) interchangeably to simplify equations depending on context.
\end{definition}

%\newpage
\section{Multivariate Bell Polynomials}
\label{sect:bell}

In order to simplify the notation for the multivariate Fa\'a di Bruno formula, we introduce a novel definition of the multivariate Bell polynomial which is analogous to the single-variate Bell polynomial in Definition \ref{def:1DBell}.

\begin{definition}[Multivariate Bell polynomials] \label{def:belldef}
Let \(\vc{j} \in \mathbb{N}^{d_1}\setminus \{\vc{0}\}\) be a vector index and let the variables \(\vc{x}_\vc{j} \in \mathbb{F}^{d_2}\) be vectors in a \(d_2\)-dimensional vector field over a field \(\mathbb{F}\). Then, given 
\(\vc{n} \in \mathbb{N}^{d_1}\), 
\(\vc{k} \in \mathbb{N}^{d_2}\), we get the solution sets \(K_{\vc{n}}\) and \(K_{\vc{n},\vc{k}}\) of functions \(\mathbb{N}^{d_1}\setminus \{\vc{0}\} \ni \vc{j} \mapsto \vc{k}_\vc{j} \in \mathbb{N}^{d_2}\)
\begin{align} 
  \label{eqn:solsetComp}
   K_{\vc{n}} 
   &= 
   \Big\{   
      \vc{k}_\vc{j} :  
      \sum_{|\vc{j}| = 1}^\infty \vc{j} |\vc{k}_\vc{j}| = \vc{n}
   \Big\}
   \\
   \label{eqn:solset}
   K_{\vc{n},\vc{k}} 
   &= 
   \Big\{   
      \vc{k}_\vc{j} :  
      \sum_{|\vc{j}| = 1}^\infty \vc{j} |\vc{k}_\vc{j}| = \vc{n}
      \textrm{ and } 
      \sum_{|\vc{j}|=1}^\infty \vc{k}_\vc{j} = \vc{k}
   \Big\}
\end{align}
which we use to define the complete and partial multivariate Bell polynomial, respectively, 
\begin{align} \label{eqn:belldef}
   \B_{\vc{n}} (\vc{x}_\vc{j} ; \vc{j})
   &= 
   \vc{n}! 
   \sum_{\vc{k}_\vc{j} \in K_{\vc{n}}} 
      \prod_{|\vc{j}| = 1}^\infty 
         \frac{1}{\vc{k}_\vc{j}!} 
         \left(   
            \frac{\vc{x}_\vc{j}}{\vc{j}!}   
         \right)^{\vc{k}_\vc{j}}
   \\ \label{eqn:belldefPart}
   \B_{\vc{n},\vc{k}} (\vc{x}_\vc{j} ; \vc{j})
   &= 
   \vc{n}! 
   \sum_{\vc{k}_\vc{j} \in K_{\vc{n},\vc{k}}} 
      \prod_{|\vc{j}| = 1}^\infty 
         \frac{1}{\vc{k}_\vc{j}!} 
         \left(   
            \frac{\vc{x}_\vc{j}}{\vc{j}!}   
         \right)^{\vc{k}_\vc{j}}.
\end{align}
\end{definition}

%\begin{proposition}
%If \(\vc{j}_0\) has a component \(j_0^i\), such that \(j_0^i > n^i\), then \(\vc{k}_{\vc{j}_0} = \vc{0}\).
%\end{proposition}

The following propositions follow directly from Definition \ref{def:belldef} and are mostly based on analogous propositions for the one-dimensional Bell polynomials found in \cite{charalambides_enumerative_2002}. We begin by finding the values of \(\vc{k}\) which make the multivariate partial Bell polynomial disappear.

\begin{proposition}
If \(|\vc{k}| > |\vc{n}|\), then \(K_{\vc{n},\vc{k}}\) is empty and \(\B_{\vc{n},\vc{k}}(\vc{x}_\vc{j} ; \vc{j}) = 0\). 
\end{proposition}
\begin{proof}
Given a solution to the equations \(\sum_{|\vc{j}| = 1}^{\infty} \vc{k}_\vc{j} = \vc{k}\)
 and  
\(\sum_{|\vc{j}| = 1}^{\infty} \vc{j} |\vc{k}_\vc{j}| = \vc{n}\), then \begin{equation*}
   |\vc{n}| 
   = 
   \sum_{|\vc{j}| = 1}^{\infty} |\vc{j}| |\vc{k}_\vc{j}| 
   \geq 
   \sum_{|\vc{j}| = 1}^{\infty} |\vc{k}_\vc{j}| 
   = 
   |\vc{k}|.
\end{equation*}
Thus no solutions exist if \(|\vc{k}|>|\vc{n}|\). 
\end{proof}

We next show that there is a maximum size of \(|\vc{j}|\) so that, if \(|\vc{j}|\) is bigger than the maximum, then \(\vc{k}_{\vc{j}}\) in Definition \ref{def:belldef} must be zero. Thus the infinite sums and products in  Definition \ref{def:belldef} are, in fact, finite.

\begin{proposition} \label{prop:maxj}
Let \( \vc{j}_0 \in \mathbb{N}^{d_1} \setminus \{\vc{0}\} \). If 
\( \vc{k}_{\vc{j}_0} \in K_{\vc{n}} \) and  
\( |\vc{j}_0|>|\vc{n}| \), then 
\(\vc{k}_{\vc{j}_0} = \vc{0}\). 
Similarly, if \(\vc{k}_{\vc{j}_0} \in K_{\vc{n},\vc{k}}\) and  \(|\vc{j}_0|>|\vc{n}| - |\vc{k}| + 1\), then \(\vc{k}_{\vc{j}_0} = \vc{0}\).
\end{proposition}
\begin{proof} We treat the two parts of this theorem seperatley.

\noindent \textbf{Part 1:} If 
\( \vc{k}_{\vc{j}_0} \in K_{\vc{n}} \) and  
\( |\vc{j}_0|>|\vc{n}| \), then 
\(\vc{k}_{\vc{j}_0} = \vc{0}\). 

We begin with \( \vc{k}_{\vc{j}_0} \in K_{\vc{n}} \) and  \( |\vc{j}_0|>|\vc{n}| \). Using the definition of the solution set \(K_\vc{n}\) in (\ref{eqn:solsetComp}), we get 
\begin{equation*}
   |\vc{n}| = \sum_{|\vc{j}| = 1}^{\infty} |\vc{j}||\vc{k}_\vc{j}|.
\end{equation*}
Note that every term in the sum on the right hand side is positive. Subtracting \(|\vc{j}_0||\vc{k}_{\vc{j}_0}|\) from both sides, we get 
\begin{equation*}
   |\vc{n}| - |\vc{j}_0||\vc{k}_{\vc{j}_0}| 
   = 
   \sum_{|\vc{j}| = 1}^{\infty} |\vc{j}||\vc{k}_\vc{j}| 
   - |\vc{j}_0||\vc{k}_{\vc{j}_0}|
   \geq 0.
\end{equation*} 
From the hypothesis, \(|\vc{j}_0| > |\vc{n}|\), so \(|\vc{k}_{\vc{j}_0}| = 0\) which only happens when \(\vc{k}_{\vc{j}_0} = \vc{0}\).

\noindent \textbf{Part 2:} If \(\vc{k}_{\vc{j}_0} \in K_{\vc{n},\vc{k}}\) and  \(|\vc{j}_0|>|\vc{n}| - |\vc{k}| + 1\), then \(\vc{k}_{\vc{j}_0} = \vc{0}\).

Let \(\vc{k}_{\vc{j}_0} \in K_{\vc{n},\vc{k}}\) and  \(|\vc{j}_0|>|\vc{n}| - |\vc{k}| + 1\). Consider \(|\vc{n}| - |\vc{k}|\). Using the definition of the solution set \(K_{\vc{n},\vc{k}}\) in (\ref{eqn:solset}), we get 
\begin{equation*}
   |\vc{n}| - |\vc{k}| 
   = 
   \sum_{|\vc{j}| = 1}^{\infty} |\vc{j}||\vc{k}_\vc{j}| 
   - 
   \sum_{|\vc{j}| = 1}^{\infty}|\vc{k}_j| 
   = 
   \sum_{|\vc{j}| = 1}^{\infty} (|\vc{j}| - 1) |\vc{k}_\vc{j}|.
\end{equation*}
Because \(|\vc{j}| \geq 1\), \(|\vc{j}| - 1 \geq 0\). Therefore, 
\begin{equation*}
   \sum_{|\vc{j}| = 1}^{\infty} (|\vc{j}| - 1) |\vc{k}_\vc{j}| 
   \geq 
   (|\vc{j}_0| - 1) |\vc{k}_{\vc{j}_0}| 
   > 
   (|\vc{n}| - |\vc{k}|)|\vc{k}_{\vc{j}_0}|
\end{equation*}
However, this means that 
\begin{equation*}
   |\vc{n}| - |\vc{k}| 
   > 
   (|\vc{n}| - |\vc{k}|)|\vc{k}_{\vc{j}_0}|,
\end{equation*}
which means \( \vc{k}_{\vc{j}_0} = \vc{0} \).
\end{proof}

Proposition \ref{prop:maxj} means that despite summing over an infinite number of terms in the sums \(\sum_{|\vc{j}| = 1}^{\infty} \vc{k}_\vc{j} = \vc{k}\) and \(\sum_{|\vc{j}| = 1}^{\infty} \vc{j} |\vc{k}_\vc{j}| = \vc{n}\), because, after a point, every \(\vc{k}_\vc{j}\) is identically \(\vc{0}\), these sums are finite.

We now relate the complete and partial multivariate Bell polynomial to one another with a proposition analogous to (\ref{eqn:1Dpartcomb}).
\begin{proposition}
Given constant \(\vc{n} \in \mathbb{N}^{d_1}\) and variables \(\vc{x}_\vc{j} \in \mathbb{N}^{d_2}\), 
\begin{equation} \label{eqn:compAndPart}
    \B_{\vc{n}} (\vc{x}_\vc{j} ; \vc{j})
    =
    \sum_{\vc{k} \in \mathbb{N}^{d_2}}
    \B_{\vc{n},\vc{k}} (\vc{x}_\vc{j} ; \vc{j})
\end{equation}
\end{proposition}
\begin{proof}
From Definition \ref{def:belldef}, we rewrite (\ref{eqn:compAndPart}) as
\begin{equation*}
   \vc{n}! 
   \sum_{\vc{k}_\vc{j} \in K_{\vc{n}}} 
      \prod_{|\vc{j}| = 1}^\infty 
         \frac{1}{\vc{k}_\vc{j}!} 
         \left(   
            \frac{\vc{x}_\vc{j}}{\vc{j}!}   
         \right)^{\vc{k}_\vc{j}}
    =
    \vc{n}!
    \sum_{\vc{k} \in \mathbb{N}^{d_2}} 
   \sum_{\vc{k}_\vc{j} \in K_{\vc{n},\vc{k}}} 
      \prod_{|\vc{j}| = 1}^\infty 
         \frac{1}{\vc{k}_\vc{j}!} 
         \left(   
            \frac{\vc{x}_\vc{j}}{\vc{j}!}   
         \right)^{\vc{k}_\vc{j}}.
\end{equation*} 
Thus we will be done if we can show that 
\begin{equation*}
   \bigcup_{\vc{k} \in \mathbb{N}^{d_2}} K_{\vc{n},\vc{k}}
   =
   K_{\vc{n}}
   \hspace{3 em}
   \textrm{and}
   \hspace{3 em}
   \bigcap_{\vc{k} \in \mathbb{N}^{d_2}} K_{\vc{n},\vc{k}}
   =
   \varnothing.
\end{equation*}
From the definition of the solution sets in Definition \ref{def:belldef}, we know that every element of \(K_{\vc{n},\vc{k}}\) satisfies the condition for inclusion in \(K_{\vc{n}}\). Thus 
\begin{equation*}
   \bigcup_{\vc{k} \in \mathbb{N}^{d_2}} K_{\vc{n},\vc{k}}
   \subseteq
   K_{\vc{n}}.
\end{equation*}
Furthermore, by Proposition \ref{prop:maxj}, for every element of \(K_{\vc{n}}\) there exists some \(\vc{k}\) such that 
\begin{equation*}
   \sum_{|\vc{j}|=1}^\infty \vc{k}_\vc{j} = \vc{k}
\end{equation*}
so there will always be a \(K_{\vc{n},\vc{k}}\) to which the \(\vc{k}_\vc{j}\) belongs. Thus 
\begin{equation*}
   \bigcup_{\vc{k} \in \mathbb{N}^{d_2}} K_{\vc{n},\vc{k}}
   =
   K_{\vc{n}}.
\end{equation*}

Finally, if \(\bigcap_{\vc{k} \in \mathbb{N}^{d_2}} K_{\vc{n},\vc{k}}\) were nonempty, there would exist some function \(\vc{k}_\vc{j}\) and constants \(\vc{k}_1\) and \(\vc{k}_2\) so that 
\begin{equation*}
   \sum_{|\vc{j}|=1}^\infty \vc{k}_\vc{j} 
   = 
   \vc{k}_1 
   \neq
   \vc{k}_2
   =
   \sum_{|\vc{j}|=1}^\infty \vc{k}_\vc{j},
\end{equation*}
which is a contradiction. Thus 
\begin{equation*}
   \bigcap_{\vc{k} \in \mathbb{N}^{d_2}} K_{\vc{n},\vc{k}}
   =
   \varnothing.
\end{equation*}
\end{proof}

We now examine the special case where both \(\vc{n}\) and \(\vc{k}\) are zero-vectors.

\begin{proposition}
If \(\vc{n} = \vc{0} \in \mathbb{N}^{d_1}\) and \(\vc{k} = \vc{0} \in \mathbb{N}^{d_2}\), then 
\begin{equation*}
   \B_{\vc{0}} (\vc{x}_\vc{j} ; \vc{j})
   =
   1
   =
   \B_{\vc{0},\vc{0}} (\vc{x}_\vc{j} ; \vc{j}).
\end{equation*}
\end{proposition}
\begin{proof}
By the definition of the solution sets \(K_\vc{0}\) and \(K_{\vc{0},\vc{0}}\) in (\ref{eqn:solsetComp}) and (\ref{eqn:solset}), we are looking for functions which satisfy 
\(\sum_{|\vc{j}| = 1}^\infty \vc{j} |\vc{k}_\vc{j}| = \vc{0}\)
 and, for \(K_{\vc{0},\vc{0}}\), 
\(\sum_{|\vc{j}|=1}^\infty \vc{k}_\vc{j} = \vc{0}\).
The only way for \(\vc{k}_\vc{j}\) to satisfy these equations is if \(\vc{k}_\vc{j}\) is identically \(\vc{0}\) for all \(\vc{j}\). Therefore 
\begin{equation*}
K_\vc{0} = \{\vc{0}\} = K_{\vc{0},\vc{0}}.
\end{equation*}
Thus, by (\ref{eqn:belldef}) and (\ref{eqn:belldefPart}), \(\B_{\vc{0}}\) and \(\B_{\vc{0},\vc{0}}\) become 
\begin{equation*}
   \sum_{\vc{k}_\vc{j} \in \{\vc{0}\}}
      \prod_{|\vc{j}| = 1}^\infty
      \frac{1}{\vc{k}_\vc{j}!}
      \left(  
         \frac{\vc{x}_\vc{j}}{\vc{j}!}  
      \right)^{\vc{k}_\vc{j}}
   =
   \prod_{|\vc{j}| = 1}^\infty
   \frac{1}{\vc{0}!}
   \left(  
      \frac{\vc{x}_\vc{j}}{\vc{j}!}  
   \right)^{\vc{0}}
   =
   1.
\end{equation*}
\end{proof}

For the remainder of the paper, we focus on the partial multivariate Bell polynomial, so we drop the ``partial'' for brevity. 

\begin{proposition} If \(\vc{a} \in \mathbb{F}^{d_1}\) and \(b \in \mathbb{F}\) are constants, then 
\begin{equation*}
\B_{\vc{n},\vc{k}}( \vc{a}^\vc{j} b \vc{x}_\vc{j} ; \vc{j} ) 
= 
\vc{a}^\vc{n} b^{|\vc{k}|}
\B_{\vc{n},\vc{k}}( \vc{x}_\vc{j} ; \vc{j} ).
\end{equation*}
\end{proposition}
\begin{proof}
We proceed directly from Definition \ref{def:belldef}. 
\begin{align*}
   \B_{\vc{n},\vc{k}}( \vc{a}^\vc{j} b \vc{x}_\vc{j} ; \vc{j} ) 
   &=
   \vc{n}! \sum_{\vc{k}_\vc{j} \in K_{\vc{n},\vc{k}}} \prod_{|\vc{j}| = 1}^\infty 
   \frac{1}{\vc{k}_\vc{j}!} 
   \left(   
      \frac{\vc{a}^\vc{j} b \vc{x}_\vc{j}}{\vc{j}!}   
   \right)^{\vc{k}_\vc{j}}
   \\
   &=
   \vc{n}! \sum_{\vc{k}_\vc{j} \in K_{\vc{n},\vc{k}}} \prod_{|\vc{j}| = 1}^\infty 
   \frac{1}{\vc{k}_\vc{j}!}
   \left(\vc{a}^\vc{j} b\right)^{|\vc{k}_\vc{j}|} 
   \left(   
      \frac{\vc{x}_\vc{j}}{\vc{j}!}   
   \right)^{\vc{k}_\vc{j}}.
\end{align*}
Using the condition that \(\sum_{|\vc{j}| = 1}^{\infty} \vc{k}_\vc{j} = \vc{k}\) and \(\sum_{|\vc{j}| = 1}^{\infty} \vc{j} |\vc{k}_\vc{j}| = \vc{n}\), we rewrite 
\(\prod_{|\vc{j}|=1}^\infty\left(\vc{a}^\vc{j} b\right)^{|\vc{k}_\vc{j}|}\) 
as 
\(\vc{a}^\vc{n} b^{|\vc{k}|}\)
which gives
\begin{align*}
   \vc{n}! \sum_{\vc{k}_\vc{j} \in K_{\vc{n},\vc{k}}} \prod_{|\vc{j}| = 1}^\infty 
   \frac{1}{\vc{k}_\vc{j}!}
   \left(\vc{a}^\vc{j} b\right)^{|\vc{k}_\vc{j}|} 
   \left(   
      \frac{\vc{x}_\vc{j}}{\vc{j}!}   
   \right)^{\vc{k}_\vc{j}}
   &=
   \vc{n}! 
   \sum_{\vc{k}_\vc{j} \in K_{\vc{n},\vc{k}}}
      \vc{a}^\vc{n} b^{|\vc{k}|} 
      \prod_{|\vc{j}| = 1}^\infty 
         \frac{1}{\vc{k}_\vc{j}!} 
         \left(   
            \frac{\vc{x}_\vc{j}}{\vc{j}!}   
         \right)^{\vc{k}_\vc{j}}
   \\
   &=
   \vc{a}^\vc{n} b^{|\vc{k}|}
   \vc{n}! 
   \sum_{\vc{k}_\vc{j} \in K_{\vc{n},\vc{k}}} 
      \prod_{|\vc{j}| = 1}^\infty 
         \frac{1}{\vc{k}_\vc{j}!} 
         \left(   
            \frac{\vc{x}_\vc{j}}{\vc{j}!}   
         \right)^{\vc{k}_\vc{j}}
   \\
   &=
   \vc{a}^\vc{n} b^{|\vc{k}|}
   \B_{\vc{n},\vc{k}}( \vc{x}_\vc{j} ; \vc{j} ).
\end{align*}
\end{proof}

The following theorem guarantees that, if \(\vc{n}\) and \(\vc{k}\) are essentially one dimensional---only one component of each is nonzero---then the multivariate Bell polynomial agrees with the one-dimensional variant in Definition \ref{def:1DBell}.
\begin{theorem}
If \(n,k \in \mathbb{N}\) and \(\vc{e}_\alpha\) and \(\vc{e}_\beta\) are unit vectors in \(\mathbb{N}^{d_1}\) and \(\mathbb{N}^{d_2}\) respectively, then 
\begin{equation*}
   \B_{n\vc{e}_\alpha,k\vc{e}_\beta}(\vc{x}_\vc{j} ; \vc{j}) 
   = 
   B_{n,k}(\vc{e}_\alpha \bigcdot \vc{x}_{j \vc{e}_\alpha} ; j)
\end{equation*}
where \(B_{n,k}\) is the single-variable Bell polynomial in Definition \ref{def:1DBell} and \(\bigcdot\) is the standard dot product.
\end{theorem}
\begin{proof}
We start by examining the conditions on \(\vc{k}_\vc{j}\). From Definition \(\ref{def:belldef}\), we know that 
\begin{equation*}
\sum_{|\vc{j}|=1}^\infty \vc{k}_\vc{j} = k \vc{e}_\beta
\hspace{2 em}\textrm{ and }\hspace{2 em}
\sum_{|\vc{j}| = 1}^\infty \vc{j} |\vc{k}_\vc{j}| = n \vc{e}_\alpha
\end{equation*}
which means that \(\vc{k}_\vc{j}\) must always point solely in the direction of \(\vc{e}_\beta\) and \(\vc{j}\) must point solely in the direction of \(\vc{e}_\alpha\). Because of these conditions, we are able to reduce the multivariate function \(\vc{k}_\vc{j}\) to be defined in terms of the single-variate function \(k_j\) so that 
\begin{equation*}
   \vc{k}_\vc{j} = \begin{cases} 
   k_j \vc{e}_\beta  &  \vc{j} = j\vc{e}_\alpha \\
   \vc{0} & \textrm{else}
   \end{cases}.
\end{equation*} 
Thus these conditions become 
\begin{equation*}
\sum_{j=1}^\infty k_j = k
\hspace{2 em}\textrm{ and }\hspace{2 em}
\sum_{j = 1}^\infty jk_j = n.
\end{equation*}
This gives us 
\begin{equation*}
   \B_{n\vc{e}_\alpha,k\vc{e}_\beta}(\vc{x}_\vc{j} ; \vc{j})
   =
   n! 
   \sum_{k_j \in K_{{n,k}}} 
      \prod_{j = 1}^\infty 
         \frac{1}{k_j!} 
         \left(   
            \frac{\vc{x}_{j\vc{e}_\alpha}}{j!}   
         \right)^{k_j\vc{e}_\beta}
\end{equation*}
which, by the multi-index operations defined in Definition \ref{def:veclen}, gives 
\begin{equation*}
   \B_{n\vc{e}_\alpha,k\vc{e}_\beta}(\vc{x}_\vc{j} ; \vc{j})
   =
   n! 
   \sum_{k_j \in K_{{n,k}}} 
      \prod_{j = 1}^\infty 
         \frac{1}{k_j!} 
         \left(   
            \frac{\vc{e}_\beta \bigcdot \vc{x}_{j\vc{e}_\alpha}}{j!}   
         \right)^{k_j}
   =
   B_{n,k}(\vc{e}_\alpha \bigcdot \vc{x}_{j \vc{e}_\alpha} ; j).
\end{equation*}
\end{proof}

%\newpage
\section{Multivariate Fa\'a di Bruno}
\label{sect:bruno}

We now turn to the motivation for defining the multivariate Bell polynomial: it simplifies the statement of the multivariate Fa\'a di Bruno formula and provides a novel proof of the same. We base our proof strategy on an analogous proof in \cite{charalambides_enumerative_2002}. Before we state the multivariate Fa\'a di Bruno formula, we first prove a preliminary result which significantly shortens the proof. 

\begin{lemma}\label{thm:genfun}
Given some constant \(\vc{u} \in \mathbb{F}^{d_2}\) and analytic multivariate function \(\vc{f}:U \subseteq \mathbb{F}^{d_1}\rightarrow\mathbb{F}^{d_2}\) so that \(\vc{f}(\vc{x}) = \sum_{|\vc{n}| = 0}^\infty \vc{f}_\vc{n} \frac{\vc{x}^\vc{n}}{\vc{n}!} \) where \(\vc{f}_\vc{n} \in \mathbb{F}^{d_2}\) are constant Taylor coefficients, then 
\begin{equation*}
\sum_{|\vc{n}| = 0}^\infty 
	\sum_{|\vc{k}| = 0}^{|\vc{n}|}
		\B_{\vc{n},\vc{k}}(
			\vc{f}_\vc{j};\vc{j}
		)
		\vc{u}^\vc{k}\frac{\vc{x}^\vc{n}}{\vc{n}!}
=
\exp( \vc{u}\bigcdot(\vc{f}(\vc{x}) - \vc{f}_\vc{0}) )
\end{equation*}
where \(\exp\) is the standard exponential function and \(\bigcdot\) is the standard dot product.
\end{lemma}
\begin{proof}
We proceed by Definition \ref{def:belldef}. 
\begin{equation*}
   \B_{\vc{n},\vc{k}}(
			\vc{f}_\vc{j};\vc{j}
		)
		\vc{u}^\vc{k}\frac{\vc{x}^\vc{n}}{\vc{n}!}
   =
   \vc{n}! 
   \sum_{\vc{k}_\vc{j} \in K_{\vc{n},\vc{k}}} 
      \left(
      \prod_{|\vc{j}|=1}^\infty 
	   \frac{1}{\vc{k}_\vc{j}!} 
	   \left(  
	      \frac{\vc{f}_\vc{j}}{\vc{j}!}  
	   \right)^{\vc{k}_\vc{j}}
      \right)
   \vc{u}^\vc{k}\frac{\vc{x}^\vc{n}}{\vc{n}!}.
\end{equation*}
Because of the definition of the solution set \(K_{\vc{n},\vc{k}}\), 
we are able to split the \(\vc{n}\) and 
\(\vc{k}\) in the exponents of \(\vc{u}^\vc{k}\) and \(\vc{x}^\vc{n}\) into \(\vc{n} = \sum_{|\vc{j}| = 1}^\infty \vc{j} |\vc{k}_\vc{j}|\) and \(\vc{k} = \sum_{|\vc{j}|=1}^\infty \vc{k}_\vc{j}\) which allows us to incorporate the  \(\vc{u}^\vc{k}\) and \(\vc{x}^\vc{n}\) terms into the sum and product as 
\begin{equation*}
\B_{\vc{n},\vc{k}}(
			\vc{f}_\vc{j};\vc{j}
		)
		\vc{u}^\vc{k}\frac{\vc{x}^\vc{n}}{\vc{n}!}
= 
\sum_{\vc{k}_\vc{j} \in K_{\vc{n},\vc{k}}} 
\prod_{|\vc{j}|=1}^\infty 
	\frac{1}{\vc{k}_\vc{j}!} 
	\left(  
		\frac{\vc{x}^{\vc{j}}}{\vc{j}!} \vc{f}_\vc{j}*\vc{u}  
	\right)^{\vc{k}_\vc{j}}
\end{equation*}
such that \(*\) denotes component-wise multiplication resulting in a vector of the same dimension. Placing the \(\B_{\vc{n},\vc{k}}(\vc{f}_\vc{j};\vc{j})\vc{u}^\vc{k}\frac{\vc{x}^\vc{n}}{\vc{n}!}\) back into the double sum whence it came, because we are summing over all \(\vc{n}\) and \(\vc{k}\), there is no restriction in our choice of function \(\vc{k}_\vc{j}\) other than that for sufficiently large values of \(|\vc{j}|\), \(\vc{k}_\vc{j} = \vc{0}\). Thus we can rewrite the sum in \(\vc{n}\) and \(\vc{k}\) as a sum of all such functions \(\vc{k}_\vc{j} \in K = \{ \vc{k}_\vc{j}: \mathbb{N}^{d_1}\setminus \{\vc{0}\} \rightarrow \mathbb{N}^{d_2} | \sum_{|\vc{j}|=0}^\infty |\vc{k}_\vc{j}| \in \mathbb{N} \}\) 
\begin{align*}
\sum_{|\vc{n}| = 0}^\infty 
	\sum_{|\vc{k}| = 0}^{|\vc{n}|}
		\B_{\vc{n},\vc{k}}(
			\vc{f}_\vc{j};\vc{j}
		)
		\vc{u}^\vc{k}\frac{\vc{x}^\vc{n}}{\vc{n}!}
&=
\sum_{|\vc{n}| = 0}^{\infty}
\sum_{\vc{k}_\vc{j} \in K_{\vc{n},\vc{k}}} 
\prod_{|\vc{j}|=1}^\infty 
	\frac{1}{\vc{k}_\vc{j}!} 
	\left(  
		\frac{\vc{x}^{\vc{j}}}{\vc{j}!} \vc{f}_\vc{j}*\vc{u}  
	\right)^{\vc{k}_\vc{j}}
\\
&=
\sum_{\vc{k}_\vc{j} \in K}
\prod_{|\vc{j}|=1}^\infty 
	\frac{1}{\vc{k}_\vc{j}!} 
	\left(  
		\frac{\vc{x}^{\vc{j}}}{\vc{j}!} \vc{f}_\vc{j}*\vc{u}  
	\right)^{\vc{k}_\vc{j}}.
\end{align*}

We would like to transpose sum and product in order to use the Taylor expansion of the exponential function to simplify the above result. Thus we must show that, for \(\fr{k} \in \mathbb{N}^{d_2}\)
\begin{equation}\label{eqn:switch}
\sum_{\vc{k}_\vc{j} \in K}
\prod_{|\vc{j}|=1}^\infty 
	\frac{1}{\vc{k}_\vc{j}!} 
	\left(  
		\frac{\vc{x}^{\vc{j}}}{\vc{j}!} \vc{f}_\vc{j}*\vc{u}  
	\right)^{\vc{k}_\vc{j}}
=
\prod_{|\vc{j}|=1}^\infty
\sum_{|\fr{k}|=0}^{\infty} 
	\frac{1}{\fr{k}!} 
	\left(  
		\frac{\vc{x}^{\vc{j}}}{\vc{j}!} \vc{f}_\vc{j}*\vc{u}  
	\right)^{\fr{k}}.
\end{equation}
In order to show equality, we need to find a correspondence between the terms on the left and right sides of (\ref{eqn:switch}). Because each factor on the right is indexed by \(\vc{j}\), there is a different value of \(\fr{k}\) for each value of \(\vc{j}\) thus there is a unique function \(\fr{k}_\vc{j}: \mathbb{N}^{d_1}\setminus \{\vc{0}\} \rightarrow \mathbb{N}^{d_2}\) that maps \(\vc{j} \mapsto \fr{k}\). Thus there is an onto map from the right hand side of (\ref{eqn:switch}) to the left hand side. However, there are terms on the right which do not appear on the left. Any term which has an infinite number of \(\fr{k} \neq \vc{0}\) does not appear on the left hand side of (\ref{eqn:switch}). Thus we must show that if there are an infinite number of indices \(\vc{j}\) so that the corresponding \(\fr{k} = \fr{k}_\vc{j}\) is non-zero, then the corresponding product vanishes. 

In this case, we want to show that, if \(J = \{\vc{j} | \fr{k}_\vc{j} \neq \vc{0}\}\) is infinite,
\begin{equation} \label{eqn:zeroProduct}
\prod_{\vc{j} \in J}
	\frac{1}{\fr{k}_\vc{j}!} 
	\left(  
		\frac{\vc{x}^{\vc{j}}}{\vc{j}!} \vc{f}_\vc{j}*\vc{u}  
	\right)^{\fr{k}_\vc{j}}
=
0.
\end{equation}
We will be done if we show that 
\(
   \frac{1}{\fr{k}_\vc{j}!} 
	\left(  
		\frac{\vc{x}^{\vc{j}}}{\vc{j}!} \vc{f}_\vc{j}*\vc{u}  
	\right)^{\fr{k}_\vc{j}}
\)
tends towards \(0\) as \(|\vc{j}|\) goes to infinity. Considering the different factors individually, because \(\fr{k}_\vc{j}! \geq 1\), we have that 
\(\frac{1}{\fr{k}_\vc{j}!} \leq 1\). 
Furthermore, \(\vc{u}\) has some number \(u_0\) such that \(|u_i| < u_0\) for all components \(u_i\) of \(\vc{u}\). Thus 
\(|\vc{f}_\vc{j}*\vc{u}| \leq u_0 |\vc{f}_\vc{j}|\). 
Finally, because \(\vc{f}\) is analytic for \(\vc{x} \in U\), the sum 
\(\sum_{|\vc{j}| = 0}^{\infty}\frac{\vc{x}^{\vc{j}}}{\vc{j}!} \vc{f}_\vc{j}\) 
converges, so 
\(\frac{\vc{x}^{\vc{j}}}{\vc{j}!} \vc{f}_\vc{j}\) 
goes to \(\vc{0}\) as \(|\vc{j}|\) tends towards infinity as does each of its components. Therefore, for sufficiently large \(|\vc{j}|\), each component of \(
   \frac{1}{\fr{k}_\vc{j}!} 
	\left(  
		\frac{\vc{x}^{\vc{j}}}{\vc{j}!} \vc{f}_\vc{j}*\vc{u}  
	\right)^{\fr{k}_\vc{j}}
\)
will be closer to \(0\) than \(\frac{1}{u_0}\), and so 
\begin{equation*}
   \left|
      \frac{1}{\fr{k}_\vc{j}!} 
	   \left(  
		   \frac{\vc{x}^{\vc{j}}}{\vc{j}!} \vc{f}_\vc{j}*\vc{u}  
	   \right)^{\fr{k}_\vc{j}}
	\right|
	\leq
	\left|
	   \left(  
	   	u_0 
		   \frac{\vc{x}^{\vc{j}}}{\vc{j}!} \vc{f}_\vc{j} 
	   \right)^{\fr{k}_\vc{j}}
	\right|
	\leq
	\max_{1 \leq i \leq d_1}
	\left|  
		u_0 
		\frac{\vc{x}^{\vc{j}}}{\vc{j}!} \vc{f}_\vc{j} \bigcdot \vc{e}_i
	\right|
\end{equation*}
where \(\max\) selects the component with the greatest absolute value. Because each component of \(\frac{\vc{x}^{\vc{j}}}{\vc{j}!} \vc{f}_\vc{j}\) tends towards \(0\) as \(|\vc{j}|\) tends towards infinity, 
\begin{equation*}
   \lim_{|\vc{j}| \rightarrow \infty}
   \frac{1}{\fr{k}_\vc{j}!} 
	   \left(  
		   \frac{\vc{x}^{\vc{j}}}{\vc{j}!} \vc{f}_\vc{j}*\vc{u}  
	   \right)^{\fr{k}_\vc{j}}
	=
	0
\end{equation*}
Therefore, because every product which has infinitely-many non-zero \(\fr{k}\) disappears, (\ref{eqn:zeroProduct}) is fulfilled and thus so too is (\ref{eqn:switch}), so we are able to transpose the sum and product.

The sum on the right hand side of (\ref{eqn:switch}) defines the function 
\begin{equation*}
   \sum_{|\vc{i}| = 0}^\infty \frac{1}{\vc{i}!} \vc{x}^\vc{i}
   = 
   \exp(\vc{1} \bigcdot \vc{x}),
\end{equation*}
so we can write 
\(
\sum_{|\vc{n}| = 0}^\infty 
	\sum_{|\vc{k}| = 0}^{|\vc{n}|}
		\B_{\vc{n},\vc{k}}(
			\vc{f}_\vc{j};\vc{j}
		)
		\vc{u}^\vc{k}\frac{\vc{x}^\vc{n}}{\vc{n}!}
\) 
as 
\begin{align*}
\sum_{|\vc{n}| = 0}^\infty 
	\sum_{|\vc{k}| = 0}^{|\vc{n}|}
		\B_{\vc{n},\vc{k}}(
			\vc{f}_\vc{j};\vc{j}
		)
		\vc{u}^\vc{k}\frac{\vc{x}^\vc{n}}{\vc{n}!}
&=
\prod_{|\vc{j}|=1}^\infty
\sum_{|\fr{k}|=0}^{\infty} 
	\frac{1}{\fr{k}!} 
	\left(  
		\frac{\vc{x}^{\vc{j}}}{\vc{j}!} \vc{f}_\vc{j}*\vc{u}  
	\right)^{\fr{k}}
\\
&=
\prod_{|\vc{j}|=1}^\infty
\exp
\left(
   \vc{1}
   \bigcdot
   \left(
   \frac{\vc{x}^{\vc{j}}}{\vc{j}!} \vc{f}_\vc{j}*\vc{u}
   \right)
\right)
\\
&=
\prod_{|\vc{j}|=1}^\infty
\exp\left(\frac{\vc{x}^{\vc{j}}}{\vc{j}!} \vc{f}_\vc{j}\bigcdot\vc{u}\right).
\end{align*}
Finally, moving the product inside the exponential as a sum completes the lemma.
\end{proof}

We are now equipped to prove a compact form of the multivariate Fa\'a di Bruno formula.

\begin{theorem}[Multivariate Fa\'a di Bruno formula]
\label{thm:faa}
Let \(U\) and \(V\) be subsets of \(\mathbb{F}^{d_1}\) and \(\mathbb{F}^{d_2}\), respectively, and let \(\vc{g}:U \rightarrow V\) and \(\vc{f}:V \rightarrow \mathbb{F}^{d_3}\) be sufficiently differentiable functions. Then the repeated application of the derivative yields
\begin{equation}\label{eqn:faa}
   \partial_\vc{x}^\vc{n} \vc{f}(\vc{g}(\vc{x})) 
   = 
   \sum_{|\vc{k}|=0}^{|\vc{n}|} \vc{f}^{(\vc{k})}(\vc{g}(\vc{x})) 
   \, 
   \B_{\vc{n},\vc{k}} ( \vc{g}^{(\vc{j})}(\vc{x});\vc{j}).
\end{equation}
\end{theorem}

\begin{proof}
We proceed by adapting a proof from \cite{charalambides_enumerative_2002} which applies to the single-variable Fa\'a di Bruno formula. We first show that the multivariate Fa\'a di Bruno formula obeys an equation similar to (\ref{eqn:faa}) for some polynomial \(P_{\vc{n},\vc{k}} (\vc{g}^{(\vc{j})}(\vc{x});\vc{j})\) which takes the components of \( \vc{g}^{(\vc{j})}(\vc{x})\) as variables. Next, we show that this polynomial must be the multivariate Bell polynomial.

\noindent \textbf{Part 1:} We proceed by induction to show that 
\begin{equation*}
\partial_\vc{x}^\vc{n} \vc{f}(\vc{g}(\vc{x})) = 
\sum_{|\vc{k}|=0}^{|\vc{n}|} 
	\vc{f}^{(\vc{k})}(\vc{g}(\vc{x})) \, P_{\vc{n},\vc{k}} (\vc{g}^{(\vc{j})}(\vc{x});\vc{j})
\end{equation*}
for some multivariate polynomial \(P\) indexed by \(\vc{n}\) and \(\vc{k}\). 

For the base case, note that 
\begin{equation*}
\partial_\vc{x}^\vc{0} \vc{f}(\vc{g}(\vc{x}))
=
\vc{f}(\vc{g}(\vc{x}))
\end{equation*}
which has polynomial term \(P_{\vc{0},\vc{0}} = 1\).

For the inductive step, if for some \(\vc{n}\),
\begin{equation*}
\partial_\vc{x}^\vc{n} \vc{f}(\vc{g}(\vc{x})) = 
\sum_{|\vc{k}|=0}^{|\vc{n}|} 
	\vc{f}^{(\vc{k})}(\vc{g}(\vc{x})) \, P_{\vc{n},\vc{k}} (\vc{g}^{(\vc{j})}(\vc{x});\vc{j}),
\end{equation*}
then for any unit multi-index \vc{e},
\begin{multline*}
   \partial_\vc{x}^{\vc{n} + \vc{e}} \vc{f}(\vc{g}(\vc{x})) = 
   \sum_{|\vc{k}|=0}^{|\vc{n}|} 
	   \left(
	     \sum_{i = 1}^{d_2} 
	     \vc{f}^{(\vc{k} + \vc{e}_i)}
	     (\vc{g}(\vc{x})) g_i^{(\vc{e})}(\vc{x})
	     \, P_{\vc{n},\vc{k}} 
	     \left(
	        \partial_\vc{x}^\vc{j}\vc{g}(\vc{x});\vc{j}
	     \right)
	   \right)
	+\\
	\vc{f}^{\vc{k}}(\vc{g}(\vc{x})) 
	\, 
	\partial_\vc{x}^\vc{e} P_{\vc{n},\vc{k}} (\vc{g}^{(\vc{j})}(\vc{x});\vc{j}).
\end{multline*}
If we rearrange the sums so that each term is preceded by \(\vc{f}^{(\vc{k})}(\vc{g}(\vc{x}))\), we have a sum that goes from \(|\vc{k}|=0\) to \(|\vc{n}|+1 = |\vc{n} + \vc{e}|\), and because we are summing and multiplying polynomials together, we get a new polynomial which depends only on \(\vc{n}\) and \(\vc{k}\). 

\noindent\textbf{Part 2:} For the second part of the proof, we show that the polynomials \(P_{\vc{n},\vc{k}}\) which appear in the first part of this proof must be the multivariate partial Bell polynomials \(\B_{\vc{n},\vc{k}}\). Because the polynomials do not depend on our choice of functions \(\vc{f}(\vc{x})\), we may choose the case 
\begin{equation*}
\vc{f}(\vc{x}) = f(\vc{x}) = \exp(\vc{u} \bigcdot (\vc{x} - \vc{x}_0))
\end{equation*}
for some \(\vc{u}, \vc{x}_0 \in \mathbb{F}^{d_2}\).
This has the Taylor expansion \begin{equation*}
   f(\vc{x})
   =
   \exp(\vc{u} \bigcdot (\vc{x} - \vc{x}_0)) 
   = 
   \sum_{|\vc{n}| = 0}^\infty 
      \vc{u}^\vc{n}
      \frac{(\vc{x} - \vc{x}_0)^\vc{n}}{\vc{n}!} .
\end{equation*}
Setting the offset term to be \(\vc{g}_0 = \vc{g}(\vc{u}_0)\) for \(\vc{u}_0 \in U\) and evaluating \(f\) at \(\vc{g}\), we have 
\begin{equation*}
   f(\vc{g}(\vc{x})) 
   = 
   \exp(\vc{u}\bigcdot (\vc{g}(\vc{x}) - \vc{g}_0))
   =
   \sum_{|\vc{n}| = 0}^\infty 
      \vc{u}^\vc{n}
      \frac{(\vc{g}(\vc{x}) - \vc{g}_0)^\vc{n}}{\vc{n}!} .
\end{equation*}
which, by Part 1, has derivatives 
\begin{equation*}
   \partial_{\vc{x}}^{\vc{n}} f(\vc{g}(\vc{x}))|_{\vc{x} = \vc{u}_0} 
   =
   \sum_{|\vc{k}| = 0}^{\vc{n}}
      \vc{u}^\vc{k}
      P_{\vc{n},\vc{k}} (\vc{g}_\vc{j};\vc{j})
\end{equation*}
such that \(\vc{g}_\vc{j} = \vc{g}^{(\vc{j})}(\vc{u}_0)\).

Separately, by Theorem \ref{thm:genfun}, we have 
\begin{equation*}
\exp( \vc{u}\bigcdot(\vc{g}(\vc{x}) - \vc{g}_\vc{0}) )
=
\sum_{|\vc{n}| = 0}^\infty 
	\sum_{|\vc{k}| = 0}^{|\vc{n}|}
		\B_{\vc{n},\vc{k}}(
			\vc{g}_\vc{j};\vc{j}
		)
		\vc{u}^\vc{k}
		\frac{\vc{x}^\vc{n}}{\vc{n}!}.
\end{equation*}
If we take the \(\vc{n}\)-th derivative and evaluate at \(\vc{x} = \vc{u}_0\), we get 
\begin{equation*}
\partial_{\vc{x}}^{\vc{n}} f(\vc{g}(\vc{x}))|_{\vc{x} = \vc{u}_0} 
=
\sum_{|\vc{k}| = 0}^{|\vc{n}|}
   \vc{u}^{\vc{k}}
	\B_{\vc{n},\vc{k}}(
		\vc{g}_\vc{j};\vc{j}
	)
	.
\end{equation*}
Therefore, equating the two expressions for \(\partial_{\vc{x}}^{\vc{n}} f(\vc{g}(\vc{x}))|_{\vc{x} = \vc{u}_0} \), we get
\begin{equation*}
   \sum_{|\vc{k}| = 0}^{|\vc{n}|}
   \vc{u}^{\vc{k}}
	\B_{\vc{n},\vc{k}}(
		\vc{g}_\vc{j};\vc{j}
	)
	=
	\sum_{|\vc{k}| = 0}^{|\vc{n}|}
   \vc{u}^{\vc{k}}
	P_{\vc{n},\vc{k}}(
		\vc{g}_\vc{j};\vc{j}
	).
\end{equation*}
The only way for this to hold for all \(\vc{u}\) and \(\vc{g}\) is if the polynomial \(P_{\vc{n},\vc{k}}\) is the multivariate Bell polynomial \(\B_{\vc{n},\vc{k}}\).
\end{proof}

\section{Conclusion}
The extension of Bell polynomials and the simplification of the multivariate Fa\'a di Bruno formula presented in this paper allows for easier manipulation of repeated derivatives of composed multivariate functions. In particular, this work can be used to find the Taylor expansion of composed multivariate functions. Future work may focus on finding an expression of the multivariate Lagrange inversion theorem which employs the multivariate Bell polynomial as well as extending this work further by finding Bell polynomials which gives the repeated derivative of the composition of more than two multivariate functions.  

\vspace{2 em}

\noindent \textbf{Acknowledgements:} The author thanks Professors Sigrun Bodine and Jacob Price at the University of Puget Sound for their assistance in finishing and presenting this work.

\vspace{1cc}

 %Fill author(s) affiliation(s), address(es) and emails here:

{\small
\noindent \textbf{Aidan Schumann}\\
Department of Mathematics and Computer Science\\
University of Puget Sound\\
1500 N. Warner St. \\
Tacoma, WA \\
98416,\\
United States.\\
E-mail: aidan@schumann.com\\
}

\end{document}